\documentclass{article}
\usepackage{amssymb,amsmath,amsthm,graphicx}

\newtheorem{theorem}{Theorem}

\newtheorem{proposition}[theorem]{Proposition}

\theoremstyle{definition}
\newtheorem{example}{Example}
\newtheorem{definition}{Definition}

\date{}

\title{\Large \textbf{Tribracket Polynomials}}

\author{
Sam Nelson\footnote{Email: Sam.Nelson@cmc.edu. Partially supported by Simons Foundation collaboration grant 702597.}
\and Fletcher Nickerson\footnote{Email: nnickerson@hmc.edu. Supported by the HMC Math Department Giovanni Borrelli Fund.}}

\begin{document}
\maketitle

\begin{abstract} We introduce a six-variable polynomial invariant of 
Niebrzydowski tribrackets analogous to quandle, rack and biquandle
polynomials. Using the subtribrackets of a tribracket, we additionally define subtribracket polynomials and establish a sufficient condition 
for isomorphic subtribrackets to have the same polynomial regardless 
of their embedding in the ambient tribracket. As an application, we
enhance the tribracket counting invariant of knots and links using 
subtribracket polynomials and provide examples to demonstrate that this 
enhancement is proper.
\end{abstract}

\parbox{5.5in} {\textsc{Keywords:} Niebrzydowski tribrackets, 
tribracket polynomials, enhancements of counting invariants
                
\smallskip
                
}

\section{Introduction}

In \cite{N}, a two-variable polynomial invariant of finite quandles known
as the \textit{quandle polynomial} was introduced. It was generalized
to the case of biquandles in \cite{N2} and racks in \cite{CN}. 
As an application, subquandle polynomials and their analogues were 
applied to define enhancements of the various counting invariants of knots 
and links associated to quandles, biquandles and racks.

Sets with ternary operations known as \textit{knot-theoretic
ternary quasigroups} were introduced in \cite{Nie1} and used to define 
knot and link invariants. Subsequent work in \cite{Nie2} presented a homology 
theory for these structures, which was used to enhance their associated 
counting invariants.

In several recent papers, the first listed author and collaborators have 
studied and generalized the knot-theoretic ternary quasigroup structure, 
also called \textit{Niebrzydowski tribrackets},
and defined new enhanced invariants of knots and related structures. In
\cite{NNS}, tribracket counting invariants were enhanced with module
structures analogous to rack and quandle modules; in \cite{NP}, 
\textit{virtual tribrackets} extended tribracket coloring invariants to
virtual knots and links, while in \cite{GNT}, \textit{Niebrzydowski algebras}
extended tribracket colorings to trivalent spatial graphs and handlebody-links.
More recently, in \cite{NP2}, a general theory of multi-tribrackets extended
tribracket colorings to various types of generalized knot theories, and in 
\cite{ANR} tribracket-brackets analogous to biquandle brackets were defined.

In this paper we introduce a notion of tribracket polynomial analogous to the 
quandle polynomial from \cite{N}. This polynomial can be conceptualized as 
quantifying the way in which trivial actions of pairs of elements on other 
elements are distributed throughout the algebraic structure, in contrast 
with the trivial action being concentrated in the identity element of a group. 
As an application, we define subtribracket
polynomial enhancements of the tribracket counting invariant. The paper is
organized as follows. In Section \ref{NT} we review the basics of tribracket
theory. In Section \ref{P} we define tribracket polynomials and compute
some examples. In Section \ref{SPE} we apply the subtribracket polynomial
construction to enhance the tribracket counting invariant for oriented classical
knots and links. We conclude in Section \ref{Q} with some questions and 
directions for future work.

\section{Niebrzydowski Tribrackets}\label{NT}

\begin{definition}
Let $X$ be a set. A (horizontal) \textit{Niebrzydowski tribracket structure} 
on $X$, also called a \textit{knot-theoretic ternary quasigroup structure} 
on $X$, is a map $[\ ,\ ,\ ]:X\times X\times X\to X$ satisfying the conditions
\begin{itemize}
\item For every $a,b,c\in X$ there are unique $x,y,z\in X$ such that
\[[a,b,x]=c, \quad [a,y,b]=c\quad \mathrm{and} \quad [z,a,b]=c\]
and
\item For every $a,b,c,d\in X$ we have
\[[c,[a,b,c],[a,c,d]]=[b,[a,b,c],[a,b,d]]=[d,[a,b,d],[a,c,d]].\]
\end{itemize}
\end{definition}

\begin{example}
Let $G$ be any group. Then the operation
\[[a,b,c]=ba^{-1}c\]
defines a tribracket structure known as the \textit{Dehn tribracket}
of $G$.
\end{example}

\begin{example}
Let $X$ be any module over the two-variable Laurent polynomial ring
$\mathbb{Z}[t^{\pm 1}, s^{\pm 1}]$. Then the operation
\[[a,b,c]=tb+sc-tsa\]
defines a tribracket operation known as an \textit{Alexander tribracket}.
\end{example}

\begin{example}
Let $X=\{1,2,\dots, n\}$. We can specify a tribracket structure on $X$
with an \textit{operation 3-tensor}, i.e. an ordered $n$-tuple of $n\times n$
matrices where the element in matrix $i$ row $j$ column $k$ is $[i,j,k]$.
For instance, the reader can verify that the operation 3-tensor
\[
\left[
\left[\begin{array}{rrr} 2 & 3 & 1 \\ 1 & 2 & 3 \\ 3 & 1 & 2 \end{array}\right],
\left[\begin{array}{rrr} 3 & 1 & 2 \\ 2 & 3 & 1 \\ 1 & 2 & 3 \end{array}\right],
\left[\begin{array}{rrr} 1 & 2 & 3 \\ 3 & 1 & 2 \\ 2 & 3 & 1 \end{array}\right]
\right]
\]
satisfies the tribracket axioms and hence defines a Niebrzydowski tribracket
structure on the set $X=\{1,2,3\}$.
\end{example}

\begin{definition}
A \textit{subtribracket} is a subset $S\subset X$ of a tribracket and is itself a tribracket
under the restriction of $[\ ,\ ,\ ]$ to $S$.
\end{definition}

We note that closure of a subset $S$ of a finite tribracket $X$ under the 
tribracket operation is both necessary and sufficient for $S$ to be a 
subtribracket.

\begin{example}
In the tribracket structure on $X=\{1,2,3\}$ specified by the 3-tensor
\[\left[\left[\begin{array}{rrr}
1 & 3 & 2 \\
2 & 1 & 3 \\
3 & 2 & 1
\end{array}\right], \left[\begin{array}{rrr}
2 & 1 & 3 \\
3 & 2 & 1 \\
1 & 3 & 2
\end{array}\right], \left[\begin{array}{rrr}
3 & 2 & 1 \\
1 & 3 & 2 \\
2 & 1 & 3
\end{array}\right]\right],\]
the singleton sets $\{1\}, \{2\}$ and $\{3\}$ each form
proper subtribrackets, while in the tribracket structure specified by
\[\left[\left[\begin{array}{rrr}
1 & 3 & 2 \\
3 & 2 & 1\\
2 & 1 & 3 \\
\end{array}\right], \left[\begin{array}{rrr}
3 & 2 & 1 \\
2 & 1 & 3 \\
1 & 3 & 2
\end{array}\right], \left[\begin{array}{rrr}
2 & 1 & 3 \\
3 & 2 & 1 \\
1 & 3 & 2 \\
\end{array}\right]\right]\]
the only proper subtribracket is $\{1\}$.
\end{example}

\begin{definition}
A map $f:X\to Y$ between tribrackets is a \textit{tribracket homomorphism} 
if for all $a,b,c\in X$ we have 
\[[f(a),f(b),f(c)]=f([a,b,c]).\]
\end{definition}

\begin{definition}
Let $f:X\to Y$ be a tribracket homomorphism. The subtribracket of $Y$ generated
by the elements $f(x)$ for all $x\in X$ is called the \textit{image} of $f$,
denoted $\mathrm{Im}(f)$.
\end{definition}

The tribracket axioms are motivated by the \textit{Reidemeister moves}
of knot theory. The idea is to assign an element of $X$ to each region in the
planar complement of an oriented knot or link diagram $D$ with the colors 
related as depicted:
\[\includegraphics{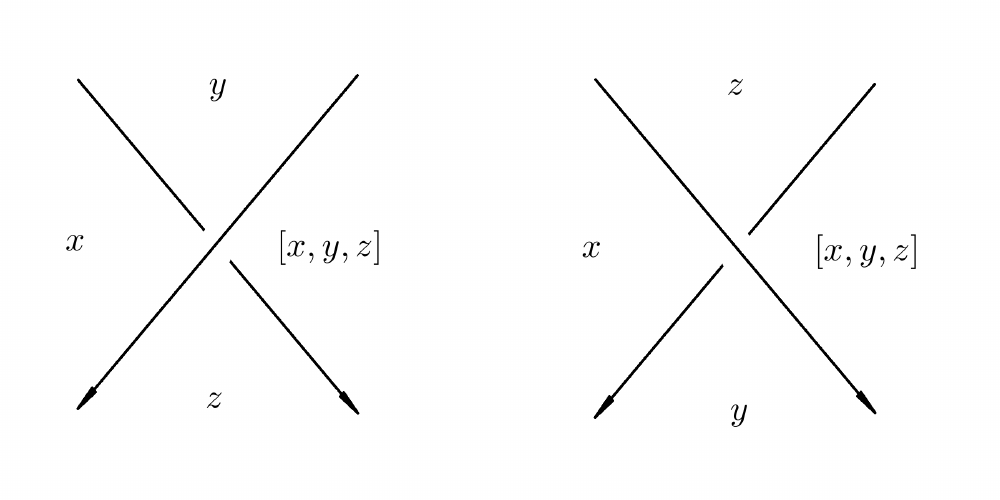}\]
Such an assignment is called an \textit{$X$-coloring} of $D$.

Then the Niebrzydowski tribracket axioms are the conditions required
to ensure that for each $X$-coloring on one side of a move, there is a
unique corresponding $X$-coloring on the other side of the move which is
unchanged outside the move's neighborhood. In particular, we have:

\begin{theorem}
Let $X$ be a finite Niebrzydowski tribracket and a $L$ be an oriented link
with diagram $D$. The cardinality of the set $\mathcal{C}(D,X)$ of 
$X$-colorings of $D$ is an integer-valued invariant of $L$, known as the
\textit{tribracket counting invariant}, denoted $\Phi_X^{\mathbb{Z}}(L)$.
\end{theorem}

\begin{example}
Let us consider the Hopf link with tribracket $X=\{1,2,3\}$ specified by 
\[
\left[
\left[\begin{array}{rrr} 1 & 2 & 3 \\ 3 & 1 & 2 \\ 2 & 3 & 1 \end{array}\right],
\left[\begin{array}{rrr} 2 & 3 & 1 \\ 1 & 2 & 3 \\ 3 & 1 & 2 \end{array}\right],
\left[\begin{array}{rrr} 3 & 1 & 2 \\ 2 & 3 & 1 \\ 1 & 2 & 3 \end{array}\right]
\right].
\]
$X$-colorings are solutions $(x,y,z,w)$ such that $[x,y,z]=w=[x,z,y]$. 
\[\includegraphics{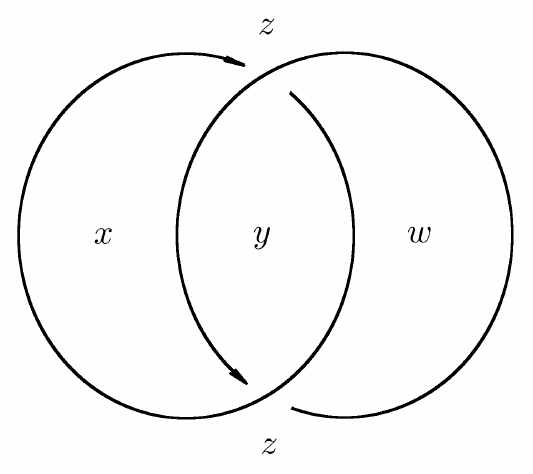}\]
The 
reader can verify that the Hopf link has nine such colorings. However, the unlink of two components has three non-interacting regions, which yield twenty-seven colorings. Hence, the Hopf link is distinguished from the 
unlink by the tribracket counting invariant for the specified tribracket $X$.
\end{example}

\section{Tribracket Polynomials}\label{P}

We can now state our main new definition.

\begin{definition}
Let $X$ be a finite Niebrzydowski tribracket. For each $i\in K$, let
\begin{itemize}
\item $l_1(i)$ be the number of elements $j\in X$ such that $[i,j,j]=i$,
\item $c_1(i)$ be the number of elements $j\in X$ such that $[j,i,j]=i$,
\item $r_1(i)$ be the number of elements $j\in X$ such that $[j,j,i]=i$,
\item $l_2(i)$ be the number of elements $j\in X$ such that $[i,j,j]=j$,
\item $c_2(i)$ be the number of elements $j\in X$ such that $[j,i,j]=j$, and
\item $r_2(i)$ be the number of elements $j\in X$ such that $[j,j,i]=j$.
\end{itemize}
Then the \textit{tribracket polynomial} of $X$ is the polynomial
\[\phi(X)=\sum_{i\in X} x^{l_1(i)}y^{c_1(i)}z^{r_1(i)}u^{l_2(i)}v^{c_2(i)}w^{r_2(i)}.\]
\end{definition}

\begin{example}
Let us compute the tribracket polynomial of the Niebrzydowski tribracket
structure on the set $X=\{1,2,3\}$ specified by the operation 3-tensor
\[
\left[
\left[\begin{array}{rrr} 
2 & 3 & 1 \\ 
1 & 2 & 3 \\ 
3 & 1 & 2 \end{array}\right],
\left[\begin{array}{rrr} 
3 & 1 & 2 \\ 
2 & 3 & 1 \\ 
1 & 2 & 3 \end{array}\right],
\left[\begin{array}{rrr} 
1 & 2 & 3 \\ 
3 & 1 & 2 \\ 
2 & 3 & 1 \end{array}\right]
\right].
\]
Starting with the element $1\in X$, we note that 
\begin{itemize}
\item the equation $[1,x,x]=1$ has no solutions, 
\item the equation $[x,1,x]=1$ has one solution, namely $x=2$, 
\item the equation $[x,x,1]=1$ has no solutions,
\item the equation $[1,x,x]=x$ has one solution, namely $x=2$,
\item the equation $[x,1,x]=x$ has one solution, namely $x=3$, and 
\item the equation $[x,x,1]=x$ has one solution, namely $x=2$.
\end{itemize}
Thus, the element $x=1$ contributes $x^0y^1z^0u^1v^1w^1=yuvw$ to the
polynomial $\phi$. Repeating for $2,3\in X$, we obtain tribracket polynomial
$\phi(X)=3yuvw$.
\end{example}

\begin{proposition}
If $f:X\to Y$ is a isomorphism of tribrackets, then $\phi(X)=\phi(Y)$.
\end{proposition}

\begin{proof}
It suffices to show
that each of the exponent variables $l_1(i)$, $c_1(i)$, $r_1(i)$, $l_2(i)$, $c_2(i)$, $r_2(i)$ used to construct the tribracket polynomial are unchanged by isomorphism $f$, meaning that the tribracket polynomials for the isomorphic tribrackets will be the same.
Let's demonstrate that the first of these variables ($l_1(i)$) has the same value for tribrackets $X$ and $Y$. Without loss of generality, this approach can be applied to the remaining five variables. 
Recall that in $X$, $l_1(i)$ counts the number of elements $j\in X$ such that $[i,j,j]=i$. Applying the definition of a tribracket homomorphism (Definition 3), our tribracket isomorphism maps this expression to $f(j)\in y$ such that $[f(i),f(j),f(j)]=f(i)$. Every element in $X$ is replaced by its image in $Y$. We know there are as many elements $j\in X$ as elements $f(j)\in Y$ because $f$ is an isomorphism. As a result, $l_1(i)$ will have the same values for the tribracket polynomials for $X$ and $Y$. The same argument can be extended for the remaining tribracket polynomial exponent variables.
\end{proof}

\begin{example}
We computed the tribracket polynomials for all tribracket structures with up to
$n=5$ elements; the results are collected in the below table.
\[\begin{array}{r|l}
n & \phi(X) \\ \hline
& \\
1 & uvwxyz  \\ 
& \\
2 &
2uvwx^2y^2z^2,\ 
2uvw
\\  & \\
3 &
3uvwxy^3z^3,\ 
3uvwx^3yz^3,\ 
3uvwx^3y^3z,\ 
u^3v^3w^3xyz + 2xyz,\ 
3uvwz,\ 
3uvwy,\ 
3uvwx
\\ & \\
4 & 4uvw,\ 
4uvwx^4y^4z^2,\ 
4uvwx^2y^2z^2,\ 
4uvwxy,\ 
4uvwz^2,\ 
4uvwx^2y^4z^4,\ 
4uvwx^4y^2z^4,\\ &  
4uvwxyz^4,\ 
4uvwx^2y^4z^2,\ 
4uvwx^4y^2z^2,\ 
4uvwy^2z^2,\ 
4uvwx^2z^2,\ 
4uvwy^2,\ 
4uvwx^2,\\ & 
4uvwx^2y^2z^4,\ 
4uvwx^2y^2,\ 
4uvwx^4y^4z^4
\\ & \\
5 & 5uvwx^5y^5z,\ 
uvw^5xyz + 4uvxyz,\ 
u^5v^5wxyz + 4wxyz,\ 
5uvwz,\ 
5uvwy,\ 
5uvwxy,\\ &  
5uvwxyz^5,\
5uvwxyz,\ 
5uvwx,\ 
5uvwx^5yz^5,\ 
5uvwxy^5z^5
\end{array}\]
\end{example}

\begin{definition}
Let $S\subset X$ be a subtribracket. Then the \textit{subtribracket polynomial}
$\phi(S\subset X)$ of $S$ with respect to $X$ is the sum
\[\phi(S\subset X)=
\sum_{i\in S} x^{l_1(i)}y^{c_1(i)}z^{r_1(i)}u^{l_2(i)}v^{c_2(i)}w^{r_2(i)}\]
of the contributions of the elements of $S$ to the tribracket polynomial.
\end{definition}

\begin{example}\label{ex:tpolylist}
The subtribracket $S=\{1\}$ of the tribracket structure on $X=\{1,2,3\}$
specified by the 3-tensor
\[\left[\left[\begin{array}{rrr}
1 & 3 & 2 \\
3 & 2 & 1\\
2 & 1 & 3 \\
\end{array}\right], \left[\begin{array}{rrr}
3 & 2 & 1 \\
2 & 1 & 3 \\
1 & 3 & 2
\end{array}\right], \left[\begin{array}{rrr}
2 & 1 & 3 \\
3 & 2 & 1 \\
1 & 3 & 2 \\
\end{array}\right]\right]\]
has subtribracket polynomial $\phi(S\subset X)=xyzu^3v^3w^3$.
\end{example}

The subtribracket polynomial of a subtribracket is not in general determined 
by the isomorphism class of the subtribracket, but carries information 
about the way the subtribracket is embedded in the ambient tribracket. 
In this way, the subtribracket polynomial is already essentially knot-theoretic.
However, there is a class of tribrackets in which isomorphic subtribrackets 
all have the same subtribracket polynomial:

\begin{definition}
A tribracket is \textit{homogeneous} if every element contributes the
same monomial to the tribracket polynomial.
\end{definition}

\begin{proposition}
Any two subtribrackets of a homogeneous tribracket with the same cardinality
have the same subtribracket polynomial.
\end{proposition}

\begin{proof}
Since every element has the same contribution, the polynomial is determined
by the number of contributions, i.e. by the the number of elements in the 
subtribracket.
\end{proof}

In particular we note that in Example \ref{ex:tpolylist}, all tribrackets 
are homogeneous for $n=1,2,4$, whereas for $n=3,5$ there are non-homogeneous 
tribrackets.

\section{Subtribracket Polynomial Enhancement}\label{SPE}

As an application, we will now define a new link invariant analogous to the
subquandle polynomial invariant  defined in \cite{N}. 

\begin{definition}
Let $X$ be a finite tribracket and $L$ an oriented classical link. The 
\textit{subtribracket polynomial enhancement} of the tribracket counting
invariant $\Phi_X(L)$ of $L$ with respect to $X$ is the multiset 
\[\Phi_{\mathrm{Im}\subset X}(L)=
\{\phi (\mathrm{Im}(f)\subset X)\ |\ f\in \mathcal{C}(L,X)\}\]
of subtribracket polynomials of the image subtribrackets over the set of
tribracket colorings of $L$ by $X$.
\end{definition}

\noindent By construction, we have our main result:

\begin{proposition}
Let $X$ be a finite tribracket. If two links $L$ and $L'$ are ambient
isotopic, then $\Phi_{\mathrm{Im}\subset X}(L)=\Phi_{\mathrm{Im}\subset X}(L')$.
\end{proposition}

\begin{example}
Our \texttt{python} computations reveal that though
the links L7a3 and L7a7 
\[\includegraphics{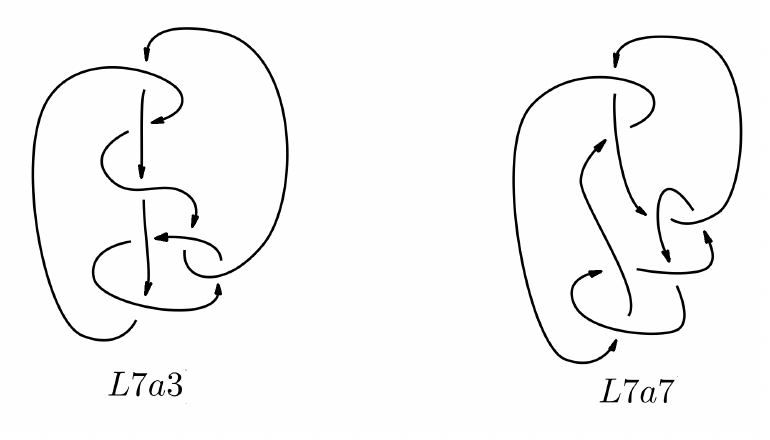}\]
both have $\Phi_x^{\mathbb{Z}}(L)=64$ colorings by the
tribracket structure on $X=\{1,2,3,4\}$ given by the 3-tensor
\[
\left[
\left[\begin{array}{rrrr}
1 & 2 & 3 & 4 \\
4 & 3 & 2 & 1 \\
3 & 4 & 1 & 2 \\
2 & 1 & 4 & 3 \\
\end{array}\right], 
\left[\begin{array}{rrrr}
4 & 3 & 2 & 1 \\
1 & 2 & 3 & 4 \\
2 & 1 & 4 & 3 \\
3 & 4 & 1 & 2 \\
\end{array}\right], 
\left[\begin{array}{rrrr}
3 & 4 & 1 & 2 \\
2 & 1 & 4 & 3 \\
1 & 2 & 3 & 4 \\
4 & 3 & 2 & 1 \\
\end{array}\right], 
\left[\begin{array}{rrrr}
2 & 1 & 4 & 3 \\
3 & 4 & 1 & 2 \\
4 & 3 & 2 & 1 \\
1 & 2 & 3 & 4
\end{array}\right]\right],
\]
they are distinguished by their subtribracket polynomials
with 
\[\Phi_{\mathrm{Im}\subset X}(L7a3)=\{4\times uvwx^2y^2z^4,\ 
48\times 4uvwx^2y^2z^4,\ 12\times 2uvwx^2y^2z^4\}\]
and
\[\Phi_{\mathrm{Im}\subset X}(L7a7)=\{4\times uvwx^2y^2z^4,\
28\times 2uvwx^2y^2z^4,\ 32\times 4uvwx^2y^2z^4\}.\]
In particular, this example shows that $\Phi_{\mathrm{Im}\subset X}$ is not
determined by $\Phi_X^{\mathbb{Z}}$ and hence is a proper enhancement.\\
\end{example}

\begin{example}
The subtribracket polynomial can distinguish single-component knots
with the same counting invariant as well. Our \texttt{python} computations
show that the granny knot $3_1\#3_1$ and the $6_1$ knot both have 2916
colorings by the Alexander tribracket $\mathbb{Z}_{18}$ with $t=5$ and $s=13$.
However, they are distinguished by their subtribracket polynomials, as \\
\begin{eqnarray*}
\Phi_{\mathrm{Im}\subset X}(3_1\#3_1) & = & 
\{324\times 3uvwy^2z^6,\ 972\times 6uvwx^{18}y^2z^6 + 12uvwy^2z^6,\ 972\times 3uvwx^{18}y^2z^6 + 6uvwy^2z^6,\\ & & \ 324\times 6uvwy^2z^6,\ 6\times uvwx^{18}y^2z^6,\ 156\times 3uvwx^{18}y^2z^6,\ 156\times 6uvwx^{18}y^2z^6,\\ & & \ 6\times 2uvwx^{18}y^2z^6\}
\end{eqnarray*}
while
\begin{eqnarray*}
\Phi_{\mathrm{Im}\subset X}(6_1) & = & \{108\times 3uvwy^2z^6,\ 1296\times 3uvwx^{18}y^2z^6 + 6uvwy^2z^6,\\ & & \ 1296\times 6uvwx^{18}y^2z^6 + 12uvwy^2z^6,\ 108\times 6uvwy^2z^6,\ 6\times uvwx^{18}y^2z^6,\\ & & \ 48\times 3uvwx^{18}y^2z^6,\ 48\times 6uvwx^{18}y^2z^6,\ 6\times 2uvwx^{18}y^2z^6\}. \\
\end{eqnarray*}
\end{example}

\begin{example}
Links that share an Alexander polynomial can also be distinguished by their subtribracket polynomials. Consider the pair of links L11n404 and L11n406. 

\[\scalebox{0.9}{\includegraphics{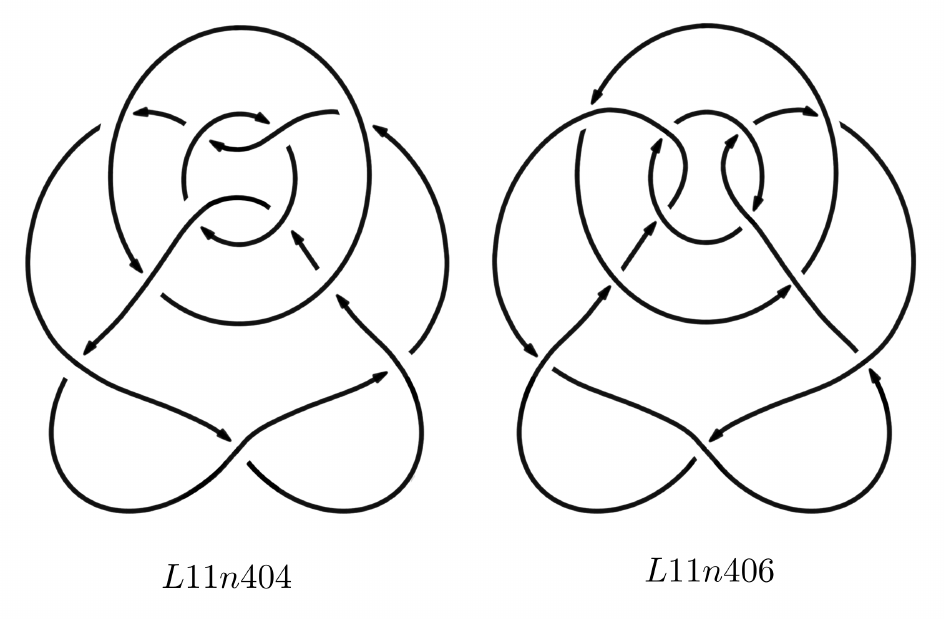}}\]

Although they both have a trivial Alexander polynomial, our subtribracket polynomial enhancement is able to differentiate them. Under the tribracket structure on $X = {1,2,3,4,5}$ specified by the tensor

\[
\left[
\left[\begin{array}{rrrrr} 1 & 3 & 2 & 5 & 4 \\ 5 & 4 & 3 & 2 & 1 \\ 4 & 2 & 5 & 1 & 3 \\ 2 & 1 & 4 & 3 & 5 \\ 3 & 5 & 1 & 4 & 2 \end{array}\right],
\left[\begin{array}{rrrrr} 5 & 4 & 3 & 2 & 1 \\ 1 & 3 & 2 & 5 & 4 \\ 1 & 3 & 2 & 5 & 4 \\ 3 & 5 & 1 & 4 & 2 \\ 4 & 2 & 5 & 1 & 3 \end{array}\right],
\left[\begin{array}{rrrrr} 4 & 2 & 5 & 1 & 3 \\ 1 & 3 & 2 & 5 & 4 \\ 3 & 5 & 1 & 4 & 2 \\ 5 & 4 & 3 & 2 & 1 \\ 2 & 1 & 4 & 3 & 5 \end{array}\right],
\left[\begin{array}{rrrrr} 2 & 1 & 4 & 3 & 5 \\ 3 & 5 & 1 & 4 & 2 \\ 5 & 4 & 3 & 2 & 1 \\ 4 & 2 & 5 & 1 & 3 \\ 1 & 3 & 2 & 5 & 4 \end{array}\right],
\left[\begin{array}{rrrrr} 3 & 5 & 1 & 4 & 2 \\ 4 & 2 & 5 & 1 & 3 \\ 2 & 1 & 4 & 3 & 5 \\ 1 & 3 & 2 & 5 & 4 \\ 5 & 4 & 3 & 2 & 1 \end{array}\right]
\right]
\]

our \texttt{python} computations reveal that they have different subtribracket polynomials:

\[\Phi_{\mathrm{Im}\subset X}(\text{L11n404})=\{1\times uvw^5xyz , 624 \times uvw^5xyz + 4 \times uvxyz\},\]

whereas

\[\Phi_{\mathrm{Im}\subset X}(\text{L11n406})=\{1\times uvw^5xyz , 124 \times uvw^5xyz + 4 \times uvxyz\}.\]

This same tribracket structure also distinguishes two orientations of the same link, denoted L10n9\{0\} and L10n9\{1\}. 

\[\scalebox{0.9}{\includegraphics{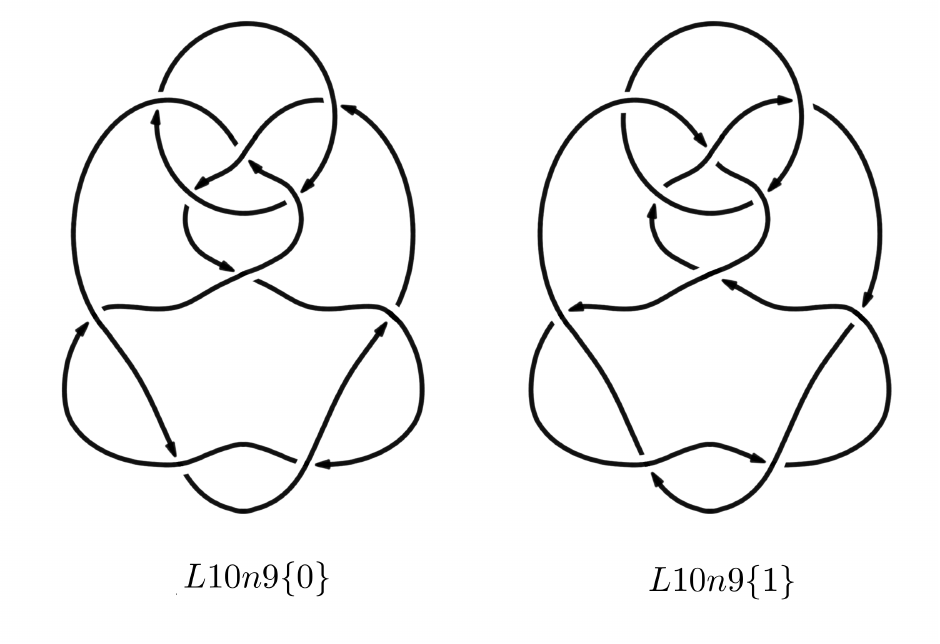}}\]

Both orientations result in the same multivariable Alexander polynomial, $1-t_1-t_2+t_1t_2$. \texttt{Python} code computed the link L10n9\{0\}'s subtribracket polynomials as

\[\Phi_{\mathrm{Im}\subset X}(\text{L10n9\{0\}})=\{1\times uvw^5xyz , 124 \times uvw^5xyz + 4 \times uvxyz\},\]

while the associated polynomials for L10n9\{1\} are

\[\Phi_{\mathrm{Im}\subset X}(\text{L10n9\{1\}})=\{1\times uvw^5xyz , 24 \times uvw^5xyz + 4 \times uvxyz\}.\]

In particular, these examples show that the subtribracket polynomial 
invariant (and indeed the tribracket counting invariant) are not determined 
by the multivariable Alexander polynomial and are sensitive to orientation.
\end{example}

\begin{example}
Using our \texttt{python} code, we computed the subtribracket polynomial
invariant for a choice of orientation for each of the prime links with up 
to seven crossings in the Thistlethwaite link table at the Knot Atlas 
\cite{KA}. These invariant values were calculated with respect to the Alexander tribracket structure on 
$X=\mathbb{Z}_8$ with parameters $t=3$ and $s=5$. The results are collected 
in the below table.
\[
\begin{array}{r|l}
L & \Phi_{\mathrm{Im}\subset X}(L) \\ \hline
L2a1 & \{8\times uvwx^8y^2z^4,\ 24\times 2uvwx^8y^2z^4,\ 32\times 4uvwx^8y^2z^4,\ 64\times 8uvwx^8y^2z^4 \}\\
L4a1 & \{8\times uvwx^8y^2z^4,\ 24\times 2uvwx^8y^2z^4,\ 96\times 4uvwx^8y^2z^4,\ 128\times 8uvwx^8y^2z^4 \}\\
L5a1 & \{8\times uvwx^8y^2z^4,\ 24\times 2uvwx^8y^2z^4,\ 96\times 4uvwx^8y^2z^4,\ 384\times 8uvwx^8y^2z^4 \}\\
L6a1 & \{8\times uvwx^8y^2z^4,\ 24\times 2uvwx^8y^2z^4,\ 96\times 4uvwx^8y^2z^4,\ 128\times 8uvwx^8y^2z^4 \}\\
L6a2 & \{8\times uvwx^8y^2z^4,\ 24\times 2uvwx^8y^2z^4,\ 32\times 4uvwx^8y^2z^4,\ 64\times 8uvwx^8y^2z^4 \}\\
L6a3 & \{8\times uvwx^8y^2z^4,\ 24\times 2uvwx^8y^2z^4,\ 32\times 4uvwx^8y^2z^4,\ 64\times 8uvwx^8y^2z^4\}\\
L6a4 & \{8\times uvwx^8y^2z^4,\ 56\times 2uvwx^8y^2z^4,\ 448\times 4uvwx^8y^2z^4,\ 512\times 8uvwx^8y^2z^4 \}\\
L6a5 & \{8\times uvwx^8y^2z^4,\ 56\times 2uvwx^8y^2z^4,\ 64\times 4uvwx^8y^2z^4,\ 128\times 8uvwx^8y^2z^4 \}\\
L6n1 & \{8\times uvwx^8y^2z^4,\ 56\times 2uvwx^8y^2z^4,\ 64\times 4uvwx^8y^2z^4,\ 128\times 8uvwx^8y^2z^4 \}\\
L7a1 & \{8\times uvwx^8y^2z^4,\ 24\times 2uvwx^8y^2z^4,\ 96\times 4uvwx^8y^2z^4,\ 384\times 8uvwx^8y^2z^4 \}\\
L7a2 & \{8\times uvwx^8y^2z^4,\ 24\times 2uvwx^8y^2z^4,\ 96\times 4uvwx^8y^2z^4,\ 128\times 8uvwx^8y^2z^4 \}\\
L7a3 & \{8\times uvwx^8y^2z^4,\ 24\times 2uvwx^8y^2z^4,\ 96\times 4uvwx^8y^2z^4,\ 384\times 8uvwx^8y^2z^4 \}\\
L7a4 & \{8\times uvwx^8y^2z^4,\ 24\times 2uvwx^8y^2z^4,\ 96\times 4uvwx^8y^2z^4,\ 384\times 8uvwx^8y^2z^4 \}\\
L7a5 & \{8\times uvwx^8y^2z^4,\ 24\times 2uvwx^8y^2z^4,\ 32\times 4uvwx^8y^2z^4,\ 64\times 8uvwx^8y^2z^4\}\\ 
L7a6 & \{8\times uvwx^8y^2z^4,\ 24\times 2uvwx^8y^2z^4,\ 32\times 4uvwx^8y^2z^4,\ 64\times 8uvwx^8y^2z^4 \}\\ 
L7a7 & \{8\times uvwx^8y^2z^4,\ 56\times 2uvwx^8y^2z^4,\ 64\times 4uvwx^8y^2z^4,\ 128\times 8uvwx^8y^2z^4 \}\\
L7n1 & \{8\times uvwx^8y^2z^4,\ 24\times 2uvwx^8y^2z^4,\ 96\times 4uvwx^8y^2z^4,\ 128\times 8uvwx^8y^2z^4 \}\\
L7n2 & \{8\times uvwx^8y^2z^4,\ 24\times 2uvwx^8y^2z^4,\ 96\times 4uvwx^8y^2z^4,\ 384\times 8uvwx^8y^2z^4\}
\end{array}
\]
\end{example}

\section{Questions}\label{Q}

We conclude with some open questions for future research.

\begin{itemize}
\item To what extent is the algebraic structure of a tribracket determined by 
its tribracket polynomial?
\item What conditions on a polynomial make it a tribracket polynomial? That is,
what conditions on a six-variable polynomial $p$ are individually necessary and 
jointly sufficient for the existence of a tribracket $X$ such that $p=\phi(x)$?
\item Does the geometry of the algebraic variety in $\mathbb{C}^6$ determined
by a tribracket polynomial or subtribracket polynomial carry any special 
significance?
\item Compared with the case of quandles with their 2-dimensional operation 
tables, the 3-dimensional operation tensors of tribrackets present a greater 
variety of options for defining tribracket polynomials. We have made one such choice here, but others are possible and might be of interest to explore.
\end{itemize}

\bibliography{sn-fn}{}
\bibliographystyle{abbrv}

\bigskip

\noindent
\textsc{Department of Mathematical Sciences \\
Claremont McKenna College \\
850 Columbia Ave. \\
Claremont, CA 91711}

\bigskip

\noindent
\textsc{Department of Mathematics \\
Harvey Mudd College\\
301 Platt Boulevard \\
Claremont, CA 91711
}

\end{document}